\documentclass[11pt,reqno]{amsart}
\usepackage{graphicx}
\usepackage{float}
\usepackage[caption = false]{subfig}
\usepackage{amsmath}
\usepackage{enumerate}
\usepackage{mathtools}
\usepackage[english]{babel}
\usepackage{amsfonts,amssymb}
\usepackage{mathrsfs}
\usepackage[utf8x]{inputenc}\usepackage[english]{babel}
\usepackage{soul}
\usepackage[all]{xy,xypic}
\usepackage{cite}
\usepackage{relsize}
\numberwithin{equation}{section}
\newtheorem{thm}{Theorem}[section]
\newtheorem{cor}[thm]{Corollary}
\newtheorem{lem}[thm]{Lemma}

\theoremstyle{definition}
\newtheorem{defn}[thm]{Definition}
\theoremstyle{remark}
\newtheorem{rem}[thm]{Remark}
\newtheorem{example}{Example}

\numberwithin{equation}{section}
\DeclareMathOperator{\RE}{Re}

\begin{document}
	
	\title[\tiny{Exact Differential Subordination}]{On Convex Dominants of Exact Differential Subordination}

	\author[S. Sivaprasad Kumar]{S. Sivaprasad Kumar}
	\address{Department of Applied Mathematics, Delhi Technological University, Delhi--110042, India}
	\email{spkumar@dce.ac.in}

		\author[S. Banga]{Shagun Banga}
	\address{Department of Applied Mathematics, Delhi Technological University, Delhi--110042, India}
	\email{shagun05banga@gmail.com}

	
\maketitle
		\begin{abstract} 	Let $h$ be a non vanishing convex univalent function and $p$ be an analytic function in $\mathbb{D}$. We consider the differential subordination $$\psi_i(p(z), z p'(z)) \prec h(z)$$ with the admissible functions in consideration as $\psi_1:=(\beta p(z)+\gamma)^{-\alpha}\left(\tfrac{(\beta p(z)+\gamma)}{\beta(1-\alpha)}+ z p'(z)\right)$ and
			$\psi_2:=\tfrac{1}{\sqrt{\gamma \beta}}\arctan\left(\sqrt{\tfrac{\beta}{\gamma}}p^{1-\alpha}(z)\right)+\left(\tfrac{1-\alpha}{\beta p^{2 (1-\alpha)}(z)+\gamma}\right)\tfrac{z p'(z)}{p^{\alpha}(z)}$. The objective of this paper is to find the dominants, preferably the best dominant(say $q$) of the solution of the above differential subordination satisfying $\psi_i(q, n zq'(z))= h(z)$. Further, we show that $\psi_i(q,zq'(z))= h(z)$ is an exact differential equation and $q$ is a convex univalent function in $\mathbb{D}$. In addition, we estimate the sharp lower bound of $\RE p$ for different choices of $h$ and derive a univalence criteria for functions in $\mathcal{H}$(class of analytic normalized functions) as an application to our results. 
	\end{abstract}
	
	\section{Introduction}
	\label{intro}
		Let  $\mathcal{H}[a,n]$ denote the class of analytic functions $f$ defined on the open unit disk $\mathbb{D}=\{z:|z|<1\}$, of the form $f(z)=a+a_n z^n+a_{n+1}z^{n+1}+\cdots$, where $n$ is a positive integer and $a \in \mathbb{C}$. Let $\mathcal{H}:=\{f \in \mathcal{H}[0,1]: f'(0)=1\}$ and $\mathcal{S}$ be the subclass of $\mathcal{H}$ consisting of the univalent functions. The Carath\'{e}odory class $\mathcal{P}$ consists of the analytic functions $p$ defined on $\mathbb{D}$ of the form $p(z) = 1 +p_1 z +p_2  z^2+\cdots,$ having positive real part in $\mathbb{D}$. Consider two analytic functions $f$ and $g$ in $\mathbb{D}$, we say $f(z)$ is subordinate to $g(z)$, denoted by $f(z) \prec g(z)$, if $f(z) = g(\omega(z))$, where $\omega$ is a Schwarz function such that $\omega(0)=0$ and $|\omega(z)|<1$. If $g$ is univalent in $\mathbb{D}$, then $f(z) \prec g(z)$ if and only if  $f(0)=g(0)$ and $f(\mathbb{D}) \subseteq g(\mathbb{D})$. Ma and Minda \cite{ma} unified many subclasses of starlike and convex functions, as follows:
	\begin{equation}\label{c}
	\mathcal{S}^*(\phi)=\left\{f \in \mathcal{S}: \dfrac{z f'(z)}{f(z)} \prec \phi(z)\right\} \text{ and }  \mathcal{C}(\phi)=\left\{f \in \mathcal{S}: 1+\dfrac{z f''(z)}{f'(z)} \prec \phi(z)\right\},
	\end{equation}
	where $\phi$ is an analytic function having positive real part in $\mathbb{D}$ and $\phi'(0)>0$ such that $\phi(\mathbb{D})$ is symmetric with respect to the real axis and starlike with respect to $\phi(0)=1$. For $\phi(z)=(1+z)/(1-z)$, the above classes, respectively reduces to $\mathcal{S}^*$, the class of starlike univalent functions and $\mathcal{C}$, the class of convex univalent functions. Recall that the $p$-valent function $f(z)=z^p+a_{p+1}z^{p+1}+\cdots$ is convex if and only if 
	\begin{equation}\label{convex}
	\RE \left(1+\dfrac{z f''(z)}{f'(z)}\right) >0,\quad z \in \mathbb{D}.
	\end{equation}This class is extensively studied in \cite{pvalent}. We may now conclude:
	\begin{defn} \label{defcon} 
		The function $f \in \mathcal{H}$ is convex univalent in $\mathbb{D}$ if and only if (\ref{convex}) holds.
	\end{defn} Moreover, the univalence criteria for such functions is illustrated recently in \cite{ozaki} as follows: 
	\begin{lem}\emph{\cite{ozaki}}\label{unilem}
		Let the function $q(z) = z+a_2 z^2 +\cdots$ be analytic in $\mathbb{D}$ such that $q(z) q'(z)/z \neq 0$ holds. Then $q$ satisfying $1+\RE\left(z q''(z)/q'(z)\right) > \alpha$ in $\mathbb{D}$ is univalent if and only if $-1/2 \leq \alpha \leq 0$. 
	\end{lem}
	Geometrically, it is also known that a function $f(z)$ is convex(starlike) function if it maps $\mathbb{D}$ onto a convex(starlike) domain. A function $f \in \mathcal{H}$ is starlike if and only if it satisfies
	\begin{equation}\label{cstar}
	\RE \left(\dfrac{z f'(z)}{f(z)}\right) >0, \quad z \in \mathbb{D}.
	\end{equation} 
	This characterization for starlikeness of functions in $\mathcal{H}$ cannot be extended to $\mathcal{H}[a,1]$. We can see this in case of $f_0(z)= (1+z)/(1-z) \in \mathcal{H}[1,1]$, which maps the unit disk $\mathbb{D}$ onto a convex domain, hence starlike in $\mathbb{D}$ but fails to satisfy~\eqref{cstar} as $\RE(z f_0'(z)/f_0(z))=\RE(2z/(1-z^2)) \not>0$. This reveals that the normalization $f(0)=0$ is essential, which is not followed by functions in $\mathcal{H}[a,1]$ $(a \neq 0)$. But, the characterization of convex functions $f \in \mathcal{H}$, given in Definition \ref{defcon} can be extended to convex functions in $\mathcal{H}[a,1]$ as follows: 
	\begin{lem}\label{extcon}
		The function $\tilde{f}(z)= a+a_1 z+ a_2 z^2+\cdots \in \mathcal{H}[a,1]$ is convex univalent in $\mathbb{D}$ if and only if it satisfies
		\begin{equation}\label{conf}
		\RE\left(1+\dfrac{z \tilde{f}''(z)}{\tilde{f}'(z)}\right) >0, \quad z \in \mathbb{D}.
		\end{equation} \end{lem}
	\begin{proof}
		Consider a function $f(z)=\tfrac{\tilde{f}(z)-a}{a_1}$, then $f \in \mathcal{H}$ and we get 
		\begin{equation}\label{comp}
		\RE\left(1+\dfrac{z \tilde{f}''(z)}{\tilde{f}'(z)}\right)=	\RE\left(1+\dfrac{z f''(z)}{f'(z)}\right).
		\end{equation}
		Now, if we assume $\tilde{f}$ is convex univalent in $\mathbb{D}$, then geometrically so is the function $f(z)$. Thus equation~\eqref{conf} follows using Definition~\ref{defcon} and equation~\eqref{comp}. Conversely, let equation (\ref{conf}) holds, then equation (\ref{convex}) holds true using equation (\ref{comp}). Thus by Definition \ref{defcon}, we get $f$ is convex univalent in $\mathbb{D}$ and geometrically so is $\tilde{f}(z)$. This completes the proof.
	\end{proof}
	The theory of  differential subordinations is extensively studied in \cite{miler2}. For details regarding dominants and best dominant of the differential subordinations, see \cite[pp. 16]{miler2}. For some more work in this direction, one may refer \cite{alids,banga, ravi}. The set of analytic and univalent functions $\tilde{q}$ on $\overline{\mathbb{D}}\backslash E(\tilde{q})$, where
	\begin{equation*}
	E(\tilde{q})=\{\zeta \in \partial \mathbb{D}: \lim_{z\rightarrow \zeta} \tilde{q}(z)= \infty\},
	\end{equation*}
	such that $\tilde{q}'(\zeta) \neq 0$ for $\zeta \in \partial \mathbb{D} \backslash E(\widetilde{q})$ is denoted by $\widetilde{Q}$. The theory of admissible functions paved a way in finding the dominants of the solutions of differential subordination altogether with a distinctive approach. The admissible function is defined as follows:
	\begin{defn}\cite{miler2}\label{adm}
		Let $\Omega$ be a set in $\mathbb{C}$, $\tilde{q} \in \widetilde{Q}$ and $n$ be a positive integer. The class of admissible functions $\Psi_n[\Omega,\tilde{q}]$, consists of the functions $\psi:\mathbb{C}^2 \times \mathbb{D}\rightarrow \mathbb{C}$,  which satisfy the admissibility condition 
		\begin{equation}\label{admcon}
		\psi(r,s;z) \notin \Omega,
		\end{equation} 
		whenever $r=\tilde{q}(\zeta)$, $s=m \zeta \tilde{q}'(\zeta)$, $z \in \mathbb{D}$, $\zeta \in \partial{ \mathbb{D}} \backslash E(\tilde{q})$ and $m \geq n$. 
	\end{defn}
	In particular, if $\Omega$ is a simply connected domain, $\Omega \neq \mathbb{C}$ and $h$ conformally maps $\mathbb{D}$ onto $\Omega$, the class $\Psi_n[\Omega,\tilde{q}]$ is denoted by $\Psi_n[h,\tilde{q}]$. Recently, in \cite{adiba}, authors have studied the applications of admissibility conditions for various known classes of starlike functions. We consider a special case for the class $\Psi_n[\Omega,\tilde{q}]$, where  $\Omega=\{w: \RE w>0\}$, $\tilde{q}(\mathbb{D})=\Omega$, $\tilde{q}(0)=1$, $E(\tilde{q})=\{1\}$ and $\tilde{q} \in \widetilde{Q}$, then above admissibility condition, given in~\eqref{admcon} reduces to 
	\begin{equation*}
	\psi(\rho i, \sigma;z) \notin \Omega,
	\end{equation*}  
	where $\rho$ and $\sigma \in \mathbb{R}$, $\sigma \leq -(n/2)|1-i \rho|^2$, $z \in \mathbb{D}$ and $n \geq 1$. In this special case, set the class $\Psi_n[\Omega,q]=:\Psi_n\{1\}$.
	
	Now, we recall that a first order differential equation of the type:
	\begin{equation}\label{diffexact}
	M(x,y) dx + N(x,y) dy =0,
	\end{equation} 
	where $M(x,y)$ and $N(x,y)$ have continuous partial derivatives in some domain $D \subset \mathbb{R}^2$, is called an exact differential equation if and only if $\tfrac{\partial M}{\partial y}=\tfrac{\partial N}{\partial x}$. The solution of such a differential equation (\ref{diffexact}), is given by:
	\begin{equation}\label{solexact}
	\int M(x,y) dx + \int \tilde{N}(y) dy =c,
	\end{equation}
	where $c$ is the integration constant and $\tilde{N}(y)$ is equal to that part of the expression $N(x,y)$, which is independent of $x$. This concept we  carry to the complex plane in the following definition:
	\begin{defn}
		Let $p \in \mathcal{H}[a_0,n]$ for some suitable $a_0$ and $\psi$ be an analytic function satisfying the following differential subordination  \begin{equation}\label{exact} \psi(p(z),zp'(z)) \prec h(z),\end{equation} which further implies $p \prec q$, where $p$ is a solution of~\eqref{exact} and $q$ is the best dominant satisfying the differential equation \begin{equation}\label{bestexact}
		\psi(q(z),nzq'(z))=h(z).\end{equation} Then the differential subordintion~\eqref{exact} is said to be an {\it exact differential subordination} if and only if the following conditions hold:\\
		$\emph{(i)}$ $\psi$ can be expressed as $\widetilde{M}(z,p)+N(z,p)~p'(z)$ such that $\tfrac{\partial \widetilde{M}}{\partial p}=\tfrac{\partial N}{\partial z}$ for all $n$.\\
		$\emph{(ii)}$ For $n=1$, the equation~\eqref{bestexact} reduces to  an exact differential equation: $$M(z,q)~dz + N(z,q)~dq=0,$$ or equivalently $\tfrac{\partial M}{\partial q}=\tfrac{\partial N}{\partial z}$, where $M(z,q):=\widetilde{M}(z,q)-h(z)$. \end{defn}
	In the present investigation, we mainly focus on the following exact type of differential subordinations:
	\begin{equation}\label{subord1}
	(\beta p(z)+\gamma)^{-\alpha}\left(\dfrac{(\beta p(z)+\gamma)}{\beta(1-\alpha)}+ z p'(z)\right) \prec h(z),
	\end{equation}
	and \begin{align}\label{subord11}
	&\dfrac{1}{\sqrt{\gamma \beta}}\arctan\left(\sqrt{\dfrac{\beta}{\gamma}}p^{1-\alpha}(z)\right)+\left(\dfrac{1-\alpha}{\beta p^{2 (1-\alpha)}(z)+\gamma}\right)\dfrac{z p'(z)}{p^{\alpha}(z)} \prec h(z),
	\end{align}
	where $h$ is a non vanishing convex univalent function in $\mathbb{D}$ with $h(0)=a$ and $p \in \mathcal{H}[a_0,n]$ for appropriate $a_0$. We denote the expression on left side of~\eqref{subord1} and~\eqref{subord11} by $\psi_1(p(z), z p'(z))$ and $\psi_2(p(z), z p'(z))$ respectively. The speciality of  the above two differential subordinations is that they generalize a result studied by Hallenbeck and Ruscheweyh \cite{halen}. Here we find the dominant and the best dominant $q$ of the solutions of the above differential subordinations,  which satisfy 
	\begin{equation}\label{qsubord}
	\psi_i(q(z), n z q'(z))=h(z)\quad (i=1,2).\end{equation}
	Now we show that the above two differential subordinations are exact. Take $\widetilde{M}(z,p):= \tfrac{(\beta p+\gamma)^{1-\alpha}}{\beta(1-\alpha)}$ and $N(z,p):=\tfrac{ z}{(\beta p+\gamma)^\alpha}$ in~\eqref{subord1}, then we have $\tfrac{\partial \widetilde{M}}{\partial p}=\tfrac{\partial N}{\partial z}$. Further, for $n=1$ the equation~(\ref{qsubord}) with $i=1$, reduces to the following exact differential equation in $z$ and $q$:
	\begin{equation}\label{exact1}
	M(z,q)~dz + N(z,q)~dq=0
	\end{equation}
	as $\tfrac{\partial M}{\partial q}=\tfrac{\partial N}{\partial z}$. Thus~\eqref{subord1} is an exact differential subordination. Since~(\ref{exact1}) is an exact differential equation, we can obtain its solution analogous to~\eqref{solexact} as follows: 
	\begin{align}\label{solexact1}
	&\int M(z,q)~ dz + \int \tilde{N}(z)~dq=c,\end{align} where $\tilde{N}(z)$ is equal to that part of the expression $N(z,q)$, which is independent of $q$. Upon simplification of the above equation, we get
	\begin{align} \label{sol1} 
	q(z)=\dfrac{\left(\dfrac{\beta(1-\alpha)}{n z^{1/n}}\int_{0}^{z}h(t) t^{(1/n)-1} dt\right)^{\tfrac{1}{1-\alpha}}-\gamma}{\beta}.
	\end{align}
	Similarly, we can show that  $\psi_2(p(z), z p'(z)) \prec h(z)$ is an exact differential subordination. Interestingly the $q$, we obtained in (\ref{sol1}), coincide with the best dominant obtained by the theory of differential subordinations \cite{miler2}.	This is illustrated in the subsequent section. Further, we show that the dominant obtained in each of the cases is convex univalent by using the concept of admissibility condition~\eqref{admcon}, which ultimately generalizes and improves some of the earlier known results. We also give some specific examples to justify our claims. In addition, we estimate the best dominant for different choices of $h$. Also, we find a criteria for univalence of functions in $\mathcal{H}$ as an application to our results. Further, using this univalence criteria, we obtain a few more results. 
	
	In what follows, let us presume convex to mean convex univalent.

	\section{Main Results} 
We need the following Hallenbeck and Ruscheweyh \cite{halen} result with $\gamma=1$  to prove our main results.     
\begin{lem}\emph{\cite{halen}}\label{halen}
	Let $h$ be convex in $\mathbb{D}$ having $h(0)=a$. If $P \in \mathcal{H}[a,n]$ satisfies
	\begin{equation*}
	P(z)+z P'(z) \prec h(z),
	\end{equation*}
	then $P(z)\prec Q(z) \prec h(z)$, where 
	\begin{align}\label{Q}
	Q(z)=\dfrac{1}{n z^{1/n}}\int_{0}^{z} h(t) t^{(1/n)-1} dt,		\end{align}
	is convex and is the best $(a,n)-$dominant.
\end{lem}
The following theorem deals in finding the best dominant of the exact type differential subordination:
\begin{thm}\label{mainthm1}
	Let $\beta(\neq0)$, $\gamma$  be complex numbers, $-1\leq \alpha \leq 0$ and $h$ be a non-vanishing convex function  with $h(0)=a$. Let $a_0=((\beta (1-\alpha)a)^{1/(1-\alpha)}-\gamma)/\beta$ and $p \in \mathcal{H}[a_0,n]$ satisfies
	\begin{equation}\label{main2}
	\dfrac{(\beta p(z)+\gamma)^{1-\alpha}}{\beta(1-\alpha)}+\dfrac{ z p'(z)}{(\beta p(z)+\gamma)^\alpha} \prec h(z),
	\end{equation}
	then $p(z) \prec q(z) \prec H(z)$, where $H(z):= ((\beta(1-\alpha)h(z))^{\tfrac{1}{1-\alpha}}-\gamma)/\beta,$ 
	\begin{equation}\label{sol}
	q(z)=	\dfrac{\left(\dfrac{\beta(1-\alpha)}{n z^{1/n}}\int_{0}^{z}h(t) t^{(1/n)-1} dt\right)^{\tfrac{1}{1-\alpha}}-\gamma}{\beta},
	\end{equation}
	are convex in $\mathbb{D}$ and $q$ is the best $(a_0,n)$- dominant.
\end{thm}
\begin{proof}
	We first prove that $q$ and $H$ are convex in $\mathbb{D}.$
	Since the function $h$ here satisfies the conditions of $h$ of Lemma~\ref{halen}, the expression~\eqref{sol} becomes 
	\begin{equation}\label{Q1}
	q(z)  =\dfrac{\left(\beta(1-\alpha)Q(z)\right)^{1/(1-\alpha)}-\gamma}{\beta},
	\end{equation} where $Q$ is given by~\eqref{Q}. Thus $q$ is well-defined and analytic in $\mathbb{D}$ as $\beta \neq 0$. Let $P(z)=\tfrac{(\beta p(z)+\gamma)^{1-\alpha}}{\beta(1-\alpha)}$, then $P \in \mathcal{H}[a,n]$ and the subordination~\eqref{main2} becomes $P(z) + zP'(z) \prec h(z)$. Now from Lemma~\ref{halen}, we have $P(z) \prec Q(z) \prec h(z)$, which implies  $(\beta p(z)+\gamma)^{1-\alpha} \prec \beta (1-\alpha)Q(z) \prec \beta (1-\alpha)h(z).$  Since $h(z)$ does not vanish in $\mathbb{D}$, we have $q'(z)\neq 0$ and therefore $ 1+z q''(z)/q'(z) =:s(z)$  is analytic in $\mathbb{D}$. Logarithmic differentiation of~\eqref{Q1} yields
	\begin{equation}\label{Qcon}
	1+\dfrac{z Q''(z)}{Q'(z)}=1+\dfrac{z q''(z)}{q'(z)} + A(z)=:\tilde{\psi}(s(z);z),
	\end{equation}
	where 
	\begin{equation*}
	A(z)=-\alpha \beta \dfrac{z q'(z)}{\beta q(z)+\gamma}=\dfrac{\alpha}{\alpha-1} \dfrac{z Q'(z)}{Q(z)}
	\end{equation*} 
	and we can write~\eqref{Qcon} as $\tilde{\psi}(r)= r +A(z)$. By Lemma~\ref{halen}, the function $Q$ is convex, thus from Lemma~\ref{extcon}, we have $\RE \tilde{\psi}(s(z)) >0$. Now to prove $q$ is convex, by Lemma \ref{extcon}, it suffices to show $\RE s(z) >0$. To accomplish this task, we use \cite[Theorem 2.3i, p. 35]{miler2} and so we need to show that $\tilde{\psi} \in \Psi_n\{1\}$. Further, by using the admissibility condition for the same, it is equivalent to prove that $\RE(\tilde{\psi}(\rho i;z)) \not> 0$. From equation (\ref{Qcon}), we obtain $$\RE(\tilde{\psi}(\rho i))= \RE(\rho i + A(z)) = \RE A(z).$$ Using the fact that $Q(z) \neq 0$ for all $z \in \mathbb{D}$, we get $zQ'(z)/Q(z)$ is analytic in $\mathbb{D}$ and since $P(z) \prec Q(z)$, we have $Q(0)=a(\neq 0)$, which implies \begin{equation}\label{Q0} \dfrac{z Q'(z)}{Q(z)}\bigg|_{z=0}=0\end{equation} and also we have $Q'(z) \neq 0 $ for all $z$, therefore the quantity, given in~\eqref{Q0} lies on either side of the imaginary axis. So clearly we have $\RE A(z) \not >0$. Thus, $q$ is convex in $\mathbb{D}$. Similarly, we can prove the function $H$ is convex in $\mathbb{D}$, since $zh'(z)/h(z)$ $(h(0)=a)$ and $zQ'(z)/Q(z)$ behave alike. We now proceed to show $p(z) \prec q(z) \prec H(z)$ and $q$ is the best $(a_0,n)$- dominant. 
	Let $\psi(r,s):= \tfrac{(\beta r+\gamma)^{1-\alpha}}{\beta (1-\alpha)}+\tfrac{s}{(\beta r+\gamma)^{\alpha}}$, which corresponds to left hand side of the subordination~\eqref{main2}. First we show $p \prec H$ by proving $\psi \in \Psi_n[h,H]$, or equivalently
	\begin{equation*}
	\psi_0:=\psi(H(\zeta), m \zeta H'(\zeta))= \dfrac{(\beta H(\zeta)+\gamma)^{1-\alpha}}{\beta(1-\alpha)}
	+\dfrac{m\zeta H'(\zeta)}{(\beta H(\zeta)+\gamma)^{\alpha}} \notin h(\mathbb{D}),	
	\end{equation*}
	whenever $|\zeta|=1$ and $m\geq n$. From hypothesis, replacing $H$ by its expression in terms of $h$, in $\psi_0$, we obtain
	\begin{equation*}
	\left|\arg\left(\dfrac{\psi_0-h(\zeta)}{\zeta h'(\zeta)}\right)\right| = |\arg m|< \pi/2.
	\end{equation*}
	As the function $h(\mathbb{D})$ is convex and $m\geq 1$, $h(\mathbb{\zeta}) \in h(\partial\mathbb{D})$ and $\zeta h'(\zeta)$ is the outer normal to $h(\partial \mathbb{D})$ at $h(\zeta)$, we arrive at the conclusion that $\psi_0 \notin h(\mathbb{D})$ and therefore we obtain $p(z) \prec H(z)$. Now, we show $p \prec q$. A simple calculation shows that $q$, given by (\ref{sol}) satisfies the following differential equation
	\begin{equation}\label{qh11}
	\dfrac{(\beta q(z)+\gamma)^{1-\alpha}}{\beta(1-\alpha)}+\dfrac{ n z q'(z)}{(\beta q(z)+\gamma)^{\alpha}} = \psi(q(z),nzq'(z))= h(z).
	\end{equation}
	Now we apply \cite[Theorem 2.3f, pp.32]{miler2} to show $q$ is the best dominant. Without loss of generality, we can assume $h$ and $q$ are analytic and univalent on $\overline{\mathbb{D}}$ and $q'(\zeta) \neq 0$ for $|\zeta| =1$, which shows that $q \in \widetilde{Q}$. To complete the proof, we now show $\psi \in \Psi_n[h,q]$. This is equivalent to show that
	\begin{equation*}
	\psi_{00}:= \psi(q(\zeta), m\zeta q'(\zeta)) = \dfrac{(\beta q(\zeta)+\gamma)^{1-\alpha}}{\beta(1-\alpha)} + \dfrac{m \zeta q'(\zeta)}{(\beta q(\zeta)+\gamma)^{\alpha}} \notin h(\mathbb{D}),
	\end{equation*}
	whenever $|\zeta|=1$ and $m \geq n$. Using equations~(\ref{sol}) and (\ref{qh11}), we obtain
	\begin{equation*}
	Q(\zeta)+\dfrac{m}{n}\left(h(\zeta)-Q(\zeta)\right)=\psi_{00}.
	\end{equation*} 
	Since we have $Q(z) \prec h(z)$ from Lemma~\ref{halen} and also we know $m/n \geq 1$, it is easy to observe that $\psi_{00} \notin h(\mathbb{D})$. This completes the proof.
\end{proof}
Let $n=1$, then the best dominant $q$, given in~\eqref{sol} is same as the solution~\eqref{sol1} obtained by solving its associated exact differential equation.
By taking $\beta=1$ and $\gamma=0$ in Theorem \ref{mainthm1}, we get the following result:
\begin{cor}\label{cor1}
	Let $-1\leq \alpha \leq 0$, $h$ be non vanishing convex function in $\mathbb{D}$ with $h(0)=a$ and $a_0=((1-\alpha)a)^{1/(1-\alpha)}$. If $p \in \mathcal{H}[a_0,n]$ satisfies
	\begin{equation}\label{main1}
	\dfrac{p^{1-\alpha}(z)}{1-\alpha}
	+\dfrac{ z p'(z)}{p^{\alpha}(z)} \prec h(z),\end{equation}
	then $p(z) \prec q(z) \prec H(z)$, where   $H(z):=((1-\alpha)h(z))^{\tfrac{1}{1-\alpha}}$,  
	\begin{equation}\label{q}
	q(z)= \left(\dfrac{(1-\alpha)}{n z^{1/n}}\int_{0}^{z}h(t) t^{(1/n)-1} dt\right)^{\tfrac{1}{1-\alpha}}.
	\end{equation}
	are convex in $\mathbb{D}$ and $q$ is the best $(a_0,n)$- dominant.
\end{cor}
In  Corollary \ref{cor1}, if we choose $\alpha=-1$ and $h$ as any convex function with $h(0)=1$ and $\RE h(z) >0$, then we obtain the following result:
\begin{cor}\label{cor2}
	Let $h$ be a convex function with $h(0)=1$ and $\RE h(z) >0$. If $p \in \mathcal{H}[1,n]$ satisfies the following
	\begin{equation*}
	p^2(z)+2 z p(z)  p'(z) \prec h(z),
	\end{equation*}
	then \begin{equation}\label{mil}
	p(z) \prec q(z) =\sqrt{Q(z)}\text{,~where } Q(z)=\dfrac{1}{n z^{1/n}}\int_{0}^{z}h(t) t^{(1/n)-1}dt.\end{equation}  The function $q$ is convex and best $(1,n)$- dominant. 
\end{cor}
\begin{rem}
	The Corollary~\ref{cor2} improves the already known result of Miller and Mocanu \cite[Theorem 3.1e., p. 77]{miler2}, by additionally establishing $q$, given in (\ref{mil}) is convex in $\mathbb{D}$ apart from being the best dominant. 
	Moreover the Corollary \ref{cor1} generalizes the known result \cite[Theorem 3.1e., p. 77]{miler2} for all $\alpha$ $(-1\leq \alpha\leq 0)$ and extends the same by proving the convexity of the function $q$ as well. 
\end{rem}
For requisite knowledge about the Hypergeometric functions, one may refer \cite[p. 5-7]{miler2}. We now obtain the following corollary for different choices of $h$ in Theorem~\ref{mainthm1}, when $\beta$ is a real  number
\begin{cor}\label{cora1}
	Let $-1 \leq \alpha \leq 0$, $\beta > 0$ and $\gamma$ be a complex number. If $p \in \mathcal{H}[(\beta (1-\alpha))^{1/(1-\alpha)}-\gamma)/\beta,n]$ and  $\psi_1(p(z),zp'(z))$, given in \eqref{subord1} satisfies any of the following:
	\begin{itemize}
		\item[(i)] $\psi_1 \prec (1+Az)/(1+Bz)$, $-1\leq B<A \leq 1$, with\\ $\lambda_1:={}_2F_1\left(1,1/n,1+1/n;B\right)- \tfrac{A}{n+1}~{}_2F_1\left(1,1+1/n,2+1/n;B\right)$\\
		\item[(ii)] $\psi_1 \prec  e^{\mu z}\text{, } (|\mu| \leq 1)$, with $\lambda_2:={}_1F_1\left(1/n,1/n+1;-\mu \right)$\\
		\item[(iii)]  $\psi_1 \prec \sqrt{1+\kappa z},$ $\kappa \in [0,1]$, with  $\lambda_3:={}_2F_1\left(-1/2,1/n,1+1/n;\kappa\right)$  \\
		\item[(iv)] $-\tfrac{\rho_2 \pi}{2}<\arg \psi_1 < \tfrac{\rho_1 \pi}{2},$ where $\rho =\tfrac{\rho_1-\rho_2}{\rho_1+\rho_2}$, $\rho'=\tfrac{\rho_1+\rho_2}{2}$, $0<\rho_1,\rho_2 \leq 1$, and $c=e^{\rho \pi i}$, with  $\lambda_4:=\mathlarger{\sum}_{k=0}^{\infty}\left(\tfrac{\left(\begin{tiny}\begin{matrix} \rho' \\ k \end{matrix}\end{tiny}\right)(-c)^k {}_2F_1(\rho',1/n+k,1+1/n+k;-1)}{1+nk}\right),$	
	\end{itemize}
	then respectively for the above parts $(i)-(iv)$, we have $\RE p(z) > \zeta_i(\alpha,\beta,\gamma,\lambda_i)$, where $$\zeta_i(\alpha,\beta,\gamma,\lambda_i)= \RE (((\beta (1-\alpha)\lambda_i)^{1/(1-\alpha)}-\gamma)/\beta)$$ with $i$ assuming appropriate integral value between $1$ and $4$.	The result is sharp.\end{cor} 
\begin{proof} Let $h(z)= (1+Az)/(1+Bz)$, $e^{\mu z}$, $\sqrt{1+\kappa z}$ and $((1+cz)/(1-z))^{\rho'}$, respectively for the parts $(i)-(iv)$. Then from Theorem \ref{mainthm1}, we have $\RE p(z) > \min_{|z|\leq 1} \RE q(z)$, where $q$ is given by (\ref{Q1}), whenever either of the above parts $(i)-(iv)$ holds. Therefore, it suffices to find the minimum value of real part of $Q$ in each of the parts, where $Q$ is given by \eqref{Q}.\\
	(i)	We have $h(z)= (1+Az)/(1+Bz)$, then from equation (\ref{Q}) we get \begin{align*}Q(z)&= \dfrac{1}{n z^{1/n}} \int_{0}^{z} \left(\dfrac{1+At}{1+Bt}\right) t^{(1/n)-1}dt\\&= \dfrac{1}{n}\int_{0}^{1}(1+Btz)^{-1}t^{(1/n)-1} dt+\dfrac{Az}{n}\int_{0}^{1}(1+Btz)^{-1} t^{1/n} dt\\&={}_2F_1\left(1,1/n,1+1/n;-Bz\right)+\dfrac{Az}{n+1}~{}_2F_1\left(1,1+1/n,2+1/n;-Bz\right).\end{align*}
	Now, we have $\min_{|z|\leq 1} \RE Q(z)= Q(-1)=\lambda_1$, where $$\lambda_1:={}_2F_1\left(1,1/n,1+1/n;B\right)- \dfrac{A}{n+1}~{}_2F_1\left(1,1+1/n,2+1/n;B\right).$$ Hence the result.\\
	(ii) We have $h(z)= \exp( \mu z)$, then from equation (\ref{Q}) we obtain \begin{align*}Q(z)&= \dfrac{1}{nz^{1/n}}\int_{0}^{z} e^{ \mu t} t^{1/n-1} dt\\&=\tfrac{1}{n}\int_{0}^{1}e^{\mu t z} t^{1/n-1}dt\\&={}_1F_1\left(1/n,1/n+1;\mu z\right).\end{align*} Clearly, $\min_{|z|\leq 1} \RE Q(z)= Q(-1) = \lambda_2$, where
	$$	\lambda_2:={}_1F_1\left(1/n,1/n+1;-\mu \right).$$ This completes the proof for part (ii).\\
	(iii) 	Let $h(z)= \sqrt{1+\kappa z}$, then from equation (\ref{Q}), we have \begin{align*}Q(z)&= \dfrac{1}{n z^{1/n}} \int_{0}^{z} \sqrt{1+\kappa t}~t^{(1/n)-1}dt\\&= \dfrac{1}{n}\int_{0}^{1}\sqrt{1+\kappa tz}~t^{(1/n)-1}dt\\&={}_2F_1\left(-1/2,1/n,1+1/n;-\kappa z\right).\end{align*}Thus, we have $\min_{|z|\leq 1} \RE Q(z)= Q(-1) = \lambda_3$, where $$\lambda_3:={}_2F_1\left(-1/2,1/n,1+1/n;\kappa\right)$$ and that completes the proof for part (iii).\\
	(iv) We have $h(z)= \left(\tfrac{1+cz}{1-z}\right)^{\rho'}$, then from equation (\ref{Q}), we have \begin{align*}Q(z)&= \dfrac{1}{n z^{1/n}} \int_{0}^{z}\left(\dfrac{1+ct}{1-t}\right)^{\rho'}~t^{(1/n)-1}dt\\&= \dfrac{1}{n z^{1/n}}\int_{0}^{z}\left(\sum_{k=0}^{\infty}\left(\begin{tiny}\begin{matrix} \rho' \\ k \end{matrix}\end{tiny}\right) c^k t^{\left(1/n+k-1\right)}(1-t)^{-\rho'}\right)dt\\&= \dfrac{1}{n z^{1/n}}\sum_{k=0}^{\infty}\left(\left(\begin{tiny}\begin{matrix} \rho' \\ k \end{matrix}\end{tiny}\right)c^k\int_{0}^{1}(tz)^{\left(1/n+k-1\right)}(1-tz)^{-\rho'}z dt\right)\\&=\dfrac{1}{n }\sum_{k=0}^{\infty}\left( \left(\begin{tiny}\begin{matrix} \rho' \\ k \end{matrix}\end{tiny}\right)(cz)^k\int_{0}^{1}t^{\left(1/n+k-1\right)}(1-tz)^{-\rho'}dt\right)\\&=\sum_{k=0}^{\infty}\left(\dfrac{\left(\begin{tiny}\begin{matrix} \rho' \\ k \end{matrix}\end{tiny}\right)(cz)^k {}_2F_1(\rho',1/n+k,1+1/n+k;z)}{1+nk}\right).\end{align*}
	Now, $\min_{|z|\leq 1} \RE Q(z)= Q(-1)=\lambda_4$, where $$\lambda_4:=\mathlarger{\sum}_{k=0}^{\infty}\left(\dfrac{ \left(\begin{tiny}\begin{matrix} \rho' \\ k \end{matrix}\end{tiny}\right)(-c)^k {}_2F_1(\rho',1/n+k,1+1/n+k;-1)}{1+nk}\right).$$ Hence the result.
\end{proof}
\begin{rem}
	If $A$ and $B$ are replaced with each other in Corollary \ref{cora1}(i), then $\min_{|z|\leq1} \RE Q(z) = Q(1)$. Further by taking $A= 2a-1$ $(a \in [0,1))$, $B=1$, $\alpha=-1$, $\beta=1$ and $\gamma=0$, the above result reduces to the result of Miller and Mocanu \cite[Corollary 3.1e.1, p. 79]{miler2}.  
\end{rem}
In the following theorem, we deal with another exact type differential subordination.
\begin{thm}\label{mainthm3}
	Let $\beta$, $\gamma$$(\neq 0)$ be complex numbers, $-1 \leq \alpha \leq 0$, $h$ be a non vanishing convex function in $\mathbb{D}$ with $h(0)=a$, $|\sqrt{\gamma \beta}h(z)|< \pi/2$ and $a_0=(\sqrt{\gamma/\beta}\tan(\sqrt{\gamma \beta} a))^{1/(1-\alpha)}$. If $p \in \mathcal{H}[a_0,n]$ satisfies  
	\begin{equation}\label{subord2}
	\dfrac{1}{\sqrt{\gamma \beta}}\arctan\left(\sqrt{\dfrac{\beta}{\gamma}}p^{1-\alpha}(z)\right)+\left(\dfrac{1-\alpha}{\beta p^{2 (1-\alpha)}(z)+\gamma}\right)\dfrac{z p'(z)}{p^{\alpha}(z)} \prec h(z),
	\end{equation}
	then $p(z) \prec q(z) \prec H(z)$, where $H(z)=(\sqrt{\gamma/\beta}\tan(\sqrt{\gamma \beta} h(z)))^{1/(1-\alpha)}$ and
	\begin{align}\label{q2}
	q(z) & =  \left(\sqrt{\dfrac{\gamma}{\beta}}\tan\left(\dfrac{\sqrt{\gamma \beta}}{n z^{1/n}}\int_{0}^{z}h(t)t^{(1/n)-1}dt\right)\right)^{\tfrac{1}{1-\alpha}}
	\end{align} are convex in $\mathbb{D}$ and $q$ is the best $(a_0,n)-$dominant.
\end{thm}
\begin{proof}
	We first show that $q$ and $H$ are convex in $\mathbb{D}$. As the function $h$ here satisfies the condition of $h$ of Lemma \ref{halen}, thus $q$ reduces to 
	\begin{equation}\label{Q2}
	q(z)= \left(\sqrt{\dfrac{\gamma}{\beta}}\tan\left(\sqrt{\gamma \beta} Q(z)\right)\right)^{\tfrac{1}{1-\alpha}},
	\end{equation} where $Q$ is given by \eqref{Q} and the function $q$ is well defined and analytic in $\mathbb{D}$. Let $P(z)= (1/\sqrt{\gamma \beta})\arctan\left(\sqrt{\tfrac{\beta}{\gamma}}p^{1-\alpha}(z)\right)$, then the subordination $(\ref{subord2})$ reduces to $P(z) + zP'(z) \prec h(z)$ and clearly $P \in \mathcal{H}[a,n]$. Now from Lemma~\ref{halen}, we have $P(z) \prec Q(z) \prec h(z)$, which is equivalent to $\arctan\left(\sqrt{\tfrac{\beta}{\gamma}}p^{1-\alpha}(z)\right) \prec \sqrt{\gamma \beta} Q (z) \prec \sqrt{\gamma \beta}h(z).$ We have $q'(z) \neq 0$ as $h(z) \neq 0$ for all $z \in \mathbb{D}$ and therefore $1+z q''(z)/q'(z)=:g(z)$ is analytic in $\mathbb{D}$. The logarithmic differentiation of the function $q$, given in (\ref{Q2}) yields
	\begin{equation}\label{Qcon2}
	1+\dfrac{z Q''(z)}{Q'(z)}=1+\dfrac{z q''(z)}{q'(z)} + B(z)=:\tilde{\psi}(g(z);z),
	\end{equation}
	where
	\begin{align}\label{B}
	B(z)&=-\left(\alpha+ 2 (1-\alpha) \beta \dfrac{ q^{2 (1-\alpha)}(z)}{\gamma+\beta q^{2 (1-\alpha)}(z)}\right) \dfrac{z q'(z)}{q(z)}\nonumber 
	\\& = \dfrac{\alpha}{\alpha-1}\left(\sum_{k=0}^{\infty}(-1)^k\dfrac{(2\sqrt{\gamma\beta}Q(z))^{2k}}{(2k+1)!}\right)^{-1}\left(\dfrac{z Q'(z)}{Q(z)}\right)\nonumber\\&\quad-2\sqrt{\gamma\beta}\tan(\sqrt{\gamma\beta}Q(z))zQ'(z).
	\end{align} 
	and we can write $\tilde{\psi}(r)= r +B(z)$. By Lemma~\ref{halen}, the function $Q$ is convex, thus from Lemma~\ref{extcon}, we have $\RE \tilde{\psi}(g(z)) >0$. In order to prove $q$ is convex, by Lemma~\ref{extcon}, it suffices to show $\RE g(z) >0$. Now, to achieve this, we use~\cite[Theorem 2.3i, p. 35]{miler2}, so then we only need to show $\tilde{\psi} \in \Psi_n\{1\}$. Further, by using the admissibility condition for the same, it is equivalent to show that $\RE(\tilde{\psi}(\rho i;z)) \not> 0$. From equation~\eqref{Qcon2}, we obtain $$\RE(\psi(\rho i))= \RE(\rho i + B(z)) = \RE B(z).$$ Now, using the fact $Q(z) \neq 0$ for all $z \in \mathbb{D}$, we get $B$ is analytic in $\mathbb{D}$. Since $P(z) \prec Q(z)$, we have $Q(z) \in \mathcal{H}[a,n]$, thus \begin{equation}\label{Q00}  \left(\sum_{k=0}^{\infty}(-1)^k\dfrac{(2\sqrt{\gamma\beta}Q(z))^{2k}}{(2k+1)!}\right)^{-1}\left(\dfrac{z Q'(z)}{Q(z)}\right)\bigg|_{z=0}=0\end{equation}and  \begin{equation}\label{Q01}\tan(\sqrt{\gamma\beta}Q(z))zQ'(z)|_{z=0}=0, \end{equation} which ensures $B(z) \equiv 0$ or lies on either side of the imaginary axis. Therefore, from equation~\eqref{B}, we have $\RE B(z) \not >0$, which proves $q$ is convex in $\mathbb{D}$. Now similarly, we can prove the function $H$ is convex in $\mathbb{D}$. We now show $p(z) \prec q(z) \prec H(z)$ and $q$ is the best $(a_0,n)$- dominant. Now, for this we assume
	$$\psi(r,s):= \dfrac{1}{\sqrt{\gamma \beta}}\arctan\left(\sqrt{\dfrac{\beta}{\gamma}}r^{1-\alpha}\right)+\dfrac{(1-\alpha) s}{r^{\alpha}(\beta r^{2 (1-\alpha)}+\gamma)},$$ which corresponds to left hand side of the subordination~\eqref{subord2}. Now first we show $p \prec H$ by proving $\psi \in \Psi_n[h,H]$, or equivalently 
	\begin{align*}
	& \dfrac{1}{\sqrt{\gamma \beta}}\arctan\left(\sqrt{\dfrac{\beta}{\gamma}}H^{1-\alpha}(\zeta)\right)+\dfrac{(1-\alpha) m \zeta H'(\zeta)}{H^{\alpha}(\zeta)(\beta H^{2 (1-\alpha)}(\zeta)+\gamma)} \\ & \quad = \psi(H(\zeta), m \zeta H'(\zeta)) =: \psi_0 \notin h(\mathbb{D}),\end{align*} whenever $|\zeta|=1$ and $m \geq n$. Now, replacing $H$ by its expression in terms of $h$ in $\psi_0$, we get
	$$\left| \arctan\left(\dfrac{\psi_0-h(\zeta)}{\zeta h'(\zeta)}\right)\right|=|\arg m| < \dfrac{\pi}{2}.$$
	Since the function $h$ is convex, $h(\zeta) \in h(\partial \mathbb{D})$ and $\zeta h'(\zeta)$ is the outer normal to $h(\partial \mathbb{D})$ at $h(\zeta)$, thus we obtain $\psi_0 \notin h(\mathbb{D})$, which further implies $p(z) \prec H(z)$. A computation shows that $q$, given by \eqref{q2} satisfies the differential equation
	\begin{align}\label{nq}
	h(z)&=\dfrac{1}{\sqrt{\gamma \beta}}\arctan\left(\sqrt{\dfrac{\beta}{\gamma}}q^{1-\alpha}(z)\right)+\dfrac{(1-\alpha) n z q'(z)}{q^{\alpha}(z)(\beta q^{2 (1-\alpha)}(z)+\gamma)}\nonumber \\ &=: \psi(q(z),n z q'(z)).
	\end{align}
	Now we apply \cite[Theorem 2.3f, p. 32]{miler2} to show $q$ is the best dominant. Without loss of generality, we can assume $h$ and $q$ are analytic and univalent on $\overline{\mathbb{D}}$ and $q'(\zeta) \neq 0$ for $|\zeta|=1$, which shows that $q \in \widetilde{Q}$. In order to complete the proof, it suffices to show $\psi \in \Psi_n[h,q]$, which is equivalent to show that
	\begin{align*}
	& \dfrac{1}{\sqrt{\gamma \beta}}\arctan\left(\sqrt{\dfrac{\beta}{\gamma}}q^{1-\alpha}(\zeta)\right)+\dfrac{(1-\alpha) m \zeta q'(\zeta)}{q^{\alpha}(\zeta)(\beta q^{2 (1-\alpha)}(\zeta)+\gamma)} \\& \quad = \psi(q(\zeta), m \zeta q'(\zeta)) =: \psi_{00}  \notin h(\mathbb{D}), 
	\end{align*}
	whenever $|\zeta|=1$ and $m \geq n$. The equations~(\ref{Q}) and~(\ref{nq}) yield
	$$Q(\zeta)+\dfrac{m}{n}(h(\zeta)-Q(\zeta)) = \psi_{00}.$$ As we have $m/n \geq 1$ and from Lemma~\ref{halen}, we have $Q(z) \prec h(z)$, therefore evidently, we get $\psi_{00} \notin h(\mathbb{D})$. This completes the proof.
\end{proof}
Taking $n=1$ in Theorem \ref{mainthm3}, equation~\eqref{nq} reduces to the differential equation (\ref{exact1}) with $M(z,q):= (1/\sqrt{\gamma \beta})\arctan\left(\sqrt{\tfrac{\beta}{\gamma}}q^{1-\alpha}\right)-h(z)$ and $N(z,q):= \tfrac{(1-\alpha)z}{q^{\alpha}(\beta q^{2 (1-\alpha)}+\gamma)}$. Further, we have $\tfrac{\partial M}{\partial q}=\tfrac{\partial N}{\partial z}$, which shows that the differential equation, given in  (\ref{nq}) is exact. Consequently, the equation~\eqref{solexact} yields the solution for equation (\ref{nq}) whenever $n=1$ as follows: 
$$ (1/\sqrt{\gamma \beta}) z \arctan\left(\sqrt{\tfrac{\beta}{\gamma}}q^{1-\alpha}(z)\right)- \int_{0}^{z}h(t) dt =0,$$
which coincides with~\eqref{q2} with $n=1$. We now obtain the following corollary for different choices of $h$ in Theorem~\ref{mainthm3}, when $\beta$ is a real  number.
\begin{cor}\label{cora2}
	Let  $-1 \leq \alpha \leq 0$, $\beta > 0$ and $\gamma(\neq 0)$ be a complex number. If $p \in \mathcal{H}[a_0,n]$, where $a_0=(\sqrt{\gamma/\beta}\tan(\sqrt{\gamma \beta} a))^{1/(1-\alpha)}$ and $\psi_2(p(z),z p'(z))$, given by \eqref{subord11} satisfies any of the following:
	\begin{itemize}
		\item[(i)] $ \psi_2 \prec (1+Az)/(1+Bz)$, where $-1\leq B<A \leq 1$ such that $|\sqrt{\gamma \beta}(1+Az)/(1+Bz)|< \pi/2$\\
		\item[(ii)] $\psi_2 \prec e^{\mu z}$, where $|\mu|\leq 1$ such that $|\sqrt{\gamma \beta}e^{\mu z}|<\pi/2$\\
		\item[(iii)] $\psi_2 \prec \sqrt{1+\kappa z}$, where $\kappa \in [0,1]$ such that $|\sqrt{\gamma \beta}\sqrt{1+\kappa z}| \leq \pi/2$\\
		\item[(iv)] $-\rho_2\pi/2 < \arg \psi_2 < \rho_1\pi/2,$ $\rho =\tfrac{\rho_1-\rho_2}{\rho_1+\rho_2}$, where $0<\rho_1,\rho_2 \leq 1$, and $c=e^{\rho \pi i}$ such that $|\sqrt{\gamma \beta}((1+cz)/(1-z))^{\rho'}|< \pi/2$, $(\rho'=(\rho_1+\rho_2)/2),$\end{itemize}
	then respectively for the above parts $(i)-(iv)$, we have $\RE p(z) > \xi_i(\alpha,\beta,\gamma,\lambda_i)$$(i=1,2,3,4)$, where $$\xi_i(\alpha,\beta,\gamma,\lambda_i)=\RE \left(\sqrt{\gamma/\beta}\tan\left(\sqrt{\gamma \beta}\lambda_i\right)\right)^{\tfrac{1}{1-\alpha}}$$ and  $\lambda_i$ is same as in the Corollary \ref{cora1}. The result is sharp.  
\end{cor}
\begin{proof}
	The proof of each part of this Corollary is on the similar lines of the proof of the Corollary \ref{cora1}. Therefore, omitted.
\end{proof}
\section{Specific Examples:}
We provide here below  some illustrations to our results.
\begin{example}\label{ex1}
	If $p \in \mathcal{H}[1,1]$ satisfies
	\begin{equation}\label{eg1}
	p(z)+z p'(z) \prec \dfrac{1+z}{1-z},
	\end{equation} 
	then $p(z) \prec q(z):=-2 \tfrac{\log(1-z)}{z}-1 \prec \tfrac{1+z}{1-z}.$ Moreover, $q$ is the best $(1,1)$- dominant and is convex in $\mathbb{D}$.
\end{example}
\begin{proof}
	By taking $h(z)=\tfrac{1+z}{1-z}$, $n=1$ and $\alpha=0$ in Corollary \ref{cor1}, we obtain $Q(z)=q(z)= -2(\log(1-z))/z-1$. Now the assertion follows at once from Corollary \ref{cor1} as $h(0)=1$ and $h(z) \neq 0$ for all $z \in \mathbb{D}$. Moreover, the quantities $zQ'(z)/Q(z)$ and $zH'(z)/H(z)$ lie on either side of the imaginary axis as proved in the proof of Theorem \ref{mainthm1}.
\end{proof}
\begin{rem}\label{rem}
	The above result can also be stated as if  $p \in \mathcal{H}[1,1]$ satisfies
	\begin{equation}\label{imp} 
	\RE  (p(z)+zp'(z))>0, \end{equation}
	then
	\begin{align*}\RE p(z) &> 2\log 2-1\\& = {}_2F_1\left(1,1,2;-1\right)- \dfrac{1}{2}~{}_2F_1\left(1,2,3;-1\right) ,\end{align*}
	which is also a special case of Corollary~\ref{cora1}(i) with $A=1$, $B=-1$, $\beta=1$, $\gamma=0$ and $\alpha=0$.
\end{rem}
\begin{example}\label{ex4}
	Let $a_0=2((2/3)^{3/4}-1)$. If $p \in \mathcal{H}[a_0,1]$ satisfies the differential subordination
	\begin{equation}\label{esubord}
	\dfrac{3(p(z)/2+1)^{4/3}}{2}+ z p'(z)(p(z)/2+1)^{1/3} \prec \exp(z),
	\end{equation}
	then $p(z) \prec q(z) \prec H(z)$, where $H(z)=2((2 e^z/3)^{3/4} -1)$ and $q(z)= 2((2 (e^z-1)/(3 z))^{3/4} -1).$ Moreover, the function $q$ is convex and is the best $(a_0,1)$- dominant.
\end{example}
\begin{proof}
	Let $h(z)=\exp(z)$, $n=1$, $\alpha=-1/3$, $\beta=1/2$ and $\gamma
	=1$ in Theorem \ref{mainthm1}, then we obtain $Q(z)=(e^z-1)/z$. Now $h$ here with the given constants satisfy the hypothesis of Theorem \ref{mainthm1}, the assertion follows from the same.  Moreover, in accordance with the proof of Theorem \ref{mainthm1}, we observe that 
	\begin{equation*}
	\RE \dfrac{z Q'(z)}{Q(z)}=\RE\left(\dfrac{1+e^z(z-1)}{e^z-1}\right) \not > 0
	\end{equation*} and similarly the quantity $\RE z H'(z)/H(z)= \RE z \not > 0$, therefore both the quantities lie on either side of the imaginary axis.
\end{proof}
\begin{example}\label{ex3}
	Let $a_0= ((0.5)\tan(0.5))^{3/5}$. If $p \in \mathcal{H}[a_0,1]$ satisfies 
	\begin{equation}\label{ex31}
	2\arctan{(2 p^{5/3}(z))}+\dfrac{20}{3}\left(\dfrac{z p^{2/3}(z) p'(z)}{1+4p^{10/3}(z)}\right) \prec \dfrac{2+z}{2-z},
	\end{equation}
	then \begin{equation}\label{q1}
	p(z) \prec q(z):= \left(\tan\left(\dfrac{-1}{2}-\dfrac{2}{z}\log\left(\dfrac{2-z}{2}\right)\right)\Big/2\right)^{3/5}.\end{equation}Moreover, the function $q$ is the best $(a_0,1)$-dominant and is convex in $\mathbb{D}$.
\end{example}
\begin{proof}
	Let $h(z)=\tfrac{2+z}{2-z}$, $n=1$, $\alpha=-2/3$, $\beta=1$ and $\gamma=1/4$ in Theorem \ref{mainthm3}, then we obtain $Q(z)=-1-((4/z)\log((2-z)/z))$ and $q(z)$ is given by (\ref{q1}). As the function $h$ and all the constants satisfy the hypothesis of Theorem \ref{mainthm3}, the result follows at once. Also, we observe that the quantity in~\eqref{B} lie on either side of the imaginary axis, which is in accordance with the proof of Theorem \ref{mainthm3}
\end{proof}
\subsection{Applications to Univalent function}

We obtain a sufficient condition for univalence of $f$ in the following theorem.
\begin{thm}\label{appthm}
	Let $f \in \mathcal{H}$. If $f$ satisfies 
	\begin{equation}\label{sec41}
	\RE(f'(z)+ z f''(z)) >0,
	\end{equation} 
	then $f$ is univalent in $\mathbb{D}$ and in fact, $\RE f' > 2\log 2-1$.
\end{thm}
\begin{proof}
	Let $p(z)= f'(z)$, then we have $p \in \mathcal{H}[1,1]$. We observe that equation (\ref{sec41}) becomes equivalent to equation (\ref{eg1}) for $p(z)=f'(z)$. Since $p$ satisfies the hypothesis of Example \ref{ex1}, we obtain  $\RE f'(z) > 2\log2-1 >0$. Thus, Noshiro-Warschawski result yields the function $f$ is univalent.
\end{proof}
The result by Miller Mocanu \cite{miler} with $n=1$ is as follows: 
\begin{cor}\emph{\cite{miler}}\label{appcor}
	Let $a \in [0,1)$ and $\eta=\eta(a)$ be defined as
	\begin{equation*}
	\eta= \left(2(1-a)\beta(1) + (2a-1)\right)^{1/2},
	\end{equation*}
	where $\beta(x)=\int_{0}^{1} \tfrac{t^{x-1}}{1+t} dt$. If $f \in \mathcal{H}$, then
	\begin{equation}\label{miller}
	\RE f'(z) >a \Rightarrow \RE \sqrt{\dfrac{f(z)}{z}} > \eta(a).
	\end{equation}
\end{cor}
\begin{cor}
	Let $f \in \mathcal{H}$ and $a= 2\log2-1$. If $f$ satisfies the equation (\ref{sec41}), then\\
	$(i)$ $\RE  \sqrt{\tfrac{f(z)}{z}} > \eta(a)$.\\
	$(ii)$ $\RE \dfrac{f(z)}{z} > a$.
\end{cor}
\begin{proof}
	(i) The proof follows directly from Theorem \ref{appthm} and Corollary \ref{appcor}.\\
	(ii) Let $p(z) =f(z)/z$, then we have $p \in \mathcal{H}[1,1]$. Thus Remark \ref{rem} yields $\RE f'(z) >0 \Rightarrow \RE(f(z)/z) > a$. Thus, the result follows now using Theorem  \ref{appthm}. 
\end{proof}


\subsection*{Acknowledgment}
The work presented here is supported by a Research Fellowship from the Department of Science and Technology, New Delhi.
	
\end{document}